\documentclass[11pt]{amsart}
\usepackage{amsmath,amssymb,mathrsfs,bm,enumerate}
\usepackage{yhmath} 

\usepackage[pagebackref, colorlinks=true, linkcolor=black, citecolor=black, urlcolor=black]{hyperref}
\usepackage{tikz-cd} 
\usepackage[all]{xy}
\tikzset{cong/.style={draw=none,edge node={node [sloped, allow upside down, auto=false]{$\simeq$}}},
         Isom/.style={draw=none,every to/.append style={edge node={node [sloped, allow upside down, auto=false]{$\simeq$}}}}}   

\usepackage{stmaryrd} 

\usepackage{xargs}                      
\usepackage[colorinlistoftodos,prependcaption 
]{todonotes}                                                                                                                 

\newcommandx{\toDoChi}[2][1=]{\todo[linecolor=blue,backgroundcolor=white,bordercolor=blue,#1]{#2}}           

\newcommandx{\todoA}[2][1=]{\todo[linecolor=green,backgroundcolor=white,bordercolor=green,#1]{#2}}          

\newcommandx{\todoN}[2][1=]{\todo[linecolor=red,backgroundcolor=white,bordercolor=red,#1]{#2}}			

\newcommand{\excise}[1]{}

\newcommand{\ZZ}{\mathbb{Z}}

\newcommand{\CC}{\mathbb{C}}

\newtheorem{thm}{Theorem}
\newtheorem{crl}{Corollary}

\newtheorem{theorem}{Theorem}[subsection]
\newtheorem{theorem*}{Theorem}
\newtheorem*{Ftheorem*}{The factorization Theorem}

\newtheorem{proposition}[theorem]{Proposition}
\newtheorem{corollary}[theorem]{Corollary}

\theoremstyle{definition}
\newtheorem{rmk}[theorem]{Remark}

\theoremstyle{definition}

\newtheorem{example}[theorem]{Example}

\makeatletter
\@namedef{subjclassname@2010}{%
\textup{\footnotesize{MSC2010}}}
\makeatother

\makeatletter
\@namedef{keywordsname}{%
\textup{\footnotesize{\it Keywords and phrases}}}
\makeatother

\begin{document}

\title{Vertex algebras of C\MakeLowercase{oh}FT-type}


\dedicatory{Dedicated to Bill Fulton on the occasion of his $80$th birthday.}

\begin{abstract}  
Representations of certain  vertex algebras, here called of CohFT-type, can be used to construct vector bundles of coinvariants and conformal blocks on moduli spaces of stable curves \cite{dgt2}.  We show that such bundles define semisimple cohomological field theories.   As an application, we give an expression for their total Chern character in terms of the fusion rules,  following the approach and computation in \cite{moppz} for  bundles given by integrable modules over affine Lie algebras. It follows that the Chern classes are tautological.  Examples and open problems are discussed.
\end{abstract}


{\tiny{
\author[C.~Damiolini]{Chiara Damiolini}
\address{Chiara Damiolini \newline \indent Department of Mathematics,  \newline \indent University of Pennsylvania, Philadelphia, PA 19104-6395
 \newline \indent {\textit{Previous address:}} Department of Mathematics, Princeton University}
\email{chiarad@sas.upenn.edu}}}
{\tiny{
\author[A.~Gibney]{Angela Gibney}
\address{Angela Gibney \newline \indent Department of Mathematics,  \newline \indent University of Pennsylvania, Philadelphia, PA 19104-6395
\newline \indent {\textit{Previous address:}} Department of Mathematics, Rutgers University}
\email{agibney@sas.upenn.edu}}}
{\tiny{
\author[N.~Tarasca]{Nicola Tarasca}
\address{Nicola Tarasca 
\newline \indent Department of Mathematics \& Applied Mathematics,
\newline \indent Virginia Commonwealth University, Richmond, VA 23284}
\email{tarascan@vcu.edu}}}

\maketitle

Vertex algebras, fundamental in a number of areas of mathematics and mathematical physics,  have recently been shown to be a source of new constructions for vector bundles on moduli of curves \cite{bzf,dgt2}.  In particular, given an $n$-tuple of modules $M^i$ over a  vertex  algebra $V$ satisfying certain natural hypotheses (stated in \S \ref{sec:CohFT}), one may construct the \textit{vector bundle of coinvariants} $\mathbb{V}_g(V; M^\bullet)$ on the moduli space $\overline{\mathcal{M}}_{g,n}$ of  $n$-pointed stable curves of genus $g$ \cite{dgt2}. The fiber at a pointed curve $(C,P_{\bullet})$ is the vector space of coinvariants, i.e., the largest quotient of $\otimes_{i=1}^n M^i$ by the action of a Lie algebra determined by $(C,P_{\bullet})$ and the vertex algebra $V$.    

Such vector bundles generalize the classical coinvariants of integrable modules over affine Lie algebras \cite{TK, tuy}.   Bundles of coinvariants from vertex algebras have much in common with their classical counterparts. For instance, both support a projectively flat logarithmic connection \cite{tuy, dgt} and satisfy factorization  \cite{tuy, dgt2},  a property that makes recursive arguments about ranks and Chern classes possible.  
 
Following  \cite{moppz}, bundles of coinvariants from integrable modules over affine Lie algebras  give cohomological field theories (CohFTs for short). Here we show the same is true for their generalizations.  We say that a vertex algebra $V$ is of \emph{CohFT-type} if $V$ satisfies the hypotheses of \S \ref{sec:CohFT}.  We prove:  
\begin{thm}
\label{thm:CohFT}
For a vertex algebra $V$ of CohFT-type, the collection consisting of the Chern characters of all vector bundles of coinvariants from finitely-generated  $V$-modules forms a semisimple cohomological field theory. 
\end{thm}

In particular, the ranks of the bundles of coinvariants form a \textit{topological quantum field theory} (TQFT),  namely, the degree zero part of the CohFT. As such, the ranks are recursively determined by the \textit{fusion rules}, that is, the dimension of spaces of coinvariants  on a three-pointed rational curve (Proposition \ref{prop:CohFTRanks}). The fusion rules have been computed in the literature for many classes of vertex algebras of CohFT-type (see \S\ref{sec:ExamplesSection} for a few examples).

In fact, the CohFTs from Theorem \ref{thm:CohFT} are determined by the fusion rules. Indeed, after work of Givental and Teleman \cite{givental2001gromov, givental2001semisimple, teleman2012structure}, a semisimple CohFT is determined by its TQFT part together with some additional structure (see also \cite{pandharipande2017cohomological}).
As in \cite{moppz}, the explicit computation  of the Atiyah algebra giving rise to the projectively flat logarithmic connection allows one to determine the recursion. As the Atiyah algebra in the case of  bundles of coinvariants from vertex algebras was determined in \cite{dgt}, one is able to extend the reconstruction of the CohFTs of coinvariants from affine Lie algebras in \cite{moppz} to the general case of vertex algebras.

Namely, following  \cite{moppz}, there exists a polynomial $P_V(a_\bullet)$ with coefficients in $H^*\!\left(\overline{\mathcal{M}}_{g,n} \right)$, explicitly given in \S\ref{sec:chernMgnBar}, such that the following holds:

\begin{crl}\label{ChernProp}
For a vertex algebra  $V$ of  CohFT-type  and an $n$-tuple  $M^\bullet$ of  simple $V$-modules with $M^i$ of conformal dimension $a_i$, the Chern character of the vector bundle of coinvariants $\mathbb{V}_g(V; M^\bullet)$ is
\[
{\rm ch}\left(  \mathbb{V}_g(V; M^\bullet)  \right) = P_V(a_\bullet) \quad \mbox{in $H^*\!\left(\overline{\mathcal{M}}_{g,n}\right)$.}
\]
\end{crl}

By Corollary \ref{ChernProp},  Chern classes of  bundles of coinvariants defined by vertex algebras of CohFT-type lie in the tautological ring of $\overline{\mathcal{M}}_{g,n}$.  
As an explicit example of the classes, the first Chern class in $\mathrm{Pic}_\mathbb{Q}(\overline{\mathcal{M}}_{g,n})$ is given by:

\begin{crl}
\label{FirstChern} 
Let  $V$ be a vertex algebra of  CohFT-type and central charge~$c$, and let $M^i$  be simple $V$-modules of conformal dimension $a_{i}$. Then 
\[
c_1\left(  \mathbb{V}_g(V; M^\bullet)  \right) =  {\rm rank} \,\mathbb{V}_g(V;M^\bullet)\left( \frac{c}{2}\lambda + \sum_{i=1}^n a_i \psi_i \right)
-b_{\rm irr} \delta_{\rm irr} - \sum_{i,I} b_{i:I} \delta_{i:I},
\]
\begin{align*}
\mbox{ with } \  b_{\rm irr} & =\sum_{W\in\mathscr{W}} a_W \cdot{\rm rank}\,\mathbb{V}_{g-1}\left(V; M^{\bullet}\otimes W \otimes W'\right) \\
\mbox{ and } \ b_{i:I} &=\sum_{W\in\mathscr{W}} a_W \cdot {\rm rank}\,\mathbb{V}_{i}\left(V;M^I\otimes W \right)\cdot {\rm rank}\,\mathbb{V}_{g-i} \left(V; M^{I^c}\otimes W' \right).
\end{align*}
\end{crl}

Here $\mathscr{W}$ is the set of finitely many simple $V$-modules; $a_W$ is the conformal dimension of a simple $V$-module $W$ (\S \ref{sec:VMod}); for $I \subseteq [n]=\{1,\dots,n\}$, we set $M^I:=\otimes_{i \in I}M^i$; and the last sum is over $i,I$ such that $i\in\{0,\dots,g\}$ and $I\subseteq [n]$, modulo the relation $(i,I)\equiv(g-i,I^c)$.  	

\smallskip

As for the first Chern class in Corollary \ref{FirstChern}, the Chern classes depend on the central charge of the vertex algebra and the conformal dimensions (or weights) of the modules.  Since $V$ is of CohFT-type, the central charge and the conformal dimensions of the modules are rational \cite{DongLiMasonModular}. 

\smallskip

\noindent {\em{Plan of paper:}} We start in \S \ref{sec:background} with some background on vertex algebras. In particular, there we describe the sheaf of coinvariants  $\mathbb{V}_g(V;M^\bullet)$ and its  dual, the sheaf of conformal blocks.   In \S \ref{sec:coinvariants} we review a selection of results on vector bundles defined by representations of vertex algebras of CohFT-type,  mainly from \cite{dgt} and \cite{dgt2}, which will be used to prove the statements above. Theorem \ref{thm:CohFT} is proved  in \S \ref{sec:PrfThmCohFT} and Corollary \ref{ChernProp}  in \S \ref{sec:chernMgnBar}.  In  \S \ref{sec:ExamplesSection} we review the invariants necessary to compute the Chern classes in several examples, including  the moonshine module vertex algebra $V^{\natural}$ and even lattice vertex algebras.  We discuss the problem for commutant and  orbifold vertex algebras, illustrating with para\-fermion vertex algebras,  and orbifolds of lattice and parafermion vertex algebras.  
 
\smallskip
 
From this and prior work, it is clear that the vector bundle of coinvariants from modules over vertex algebras of CohFT-type have a number of properties in common with their classical analogues, for which much has already been discovered.  For instance, bundles of coinvariants defined from modules over affine Lie algebras are particularly interesting on $\overline{\mathcal{M}}_{0,n}$, where they are globally generated, and their sections define morphisms \cite{fakhr}.  In particular, by studying their Chern classes one can learn about the maps they define.  In \S \ref{Necessary} we discuss questions one might explore with this in mind.


\section{Background} \label{sec:background}

Here we briefly review  vertex algebras, their modules,  related Lie algebras, and the vector spaces of coinvariants they define. We  refer the reader to \cite{fhl, lepli, bzf, dgt, dgt2} for  details. We use   notation as in \cite{dgt2}, where further information and references on these topics can be found.

\subsection{The Virasoro algebra} \label{sec:Virasoro}
The \emph{Witt (Lie) algebra} $\textrm{Der}\,\mathcal{K}$ is the Lie algebra  $\mathbb{C}(\!( z)\!) \partial_z$ generated by $L_p:=-z^{p+1}\partial_z$, for $p\in \mathbb{Z}$, with Lie bracket given by $[L_p, L_q] = (p-q) L_{p+q}$.

The \emph{Virasoro (Lie) algebra} $\textrm{Vir}$ is a central extension of $\textrm{Der}\,\mathcal{K}$ which is generated by a formal vector $K$ and the elements $L_p$, for $p\in \mathbb{Z}$, with Lie bracket given by
\[
[K, L_p]=0, \qquad [L_p, L_q] = (p-q) L_{p+q} + \frac{1}{12} (p^3-p)\delta_{p+q,0}\,K.
\]
A representation of $\textrm{Vir}$ has \emph{central charge} $c\in \mathbb{C}$ if 
$K\in \textrm{Vir}$ acts as $c \cdot \textrm{id}$.

\subsection{Vertex operator algebras}\label{sec:CVADef}
A \textit{vertex operator algebra} is a four-tuple  $\left(V, \bm{1}^{V}, \omega, Y(\cdot,z)\right)$ with: $V=\oplus_{i \geq 0} V_i$  a $\mathbb{Z}_{\geq 0}$-graded $\mathbb{C}$-vector space with \mbox{$\dim V_i<\infty$};  two distinguished elements  $\bm{1}^{V} \in V_0$ (the \emph{vacuum vector}) and $\omega \in V_2$ (the \emph{conformal vector});  a linear map 
\mbox{$Y(\cdot,z)\colon V \rightarrow  \textrm{End}(V)\left\llbracket z,z^{-1} \right\rrbracket$}  that assigns to  $A \in V$ the \textit{vertex operator} $Y(A,z) :=\sum_{i\in\mathbb{Z}} A_{(i)}z^{-i-1}$. 
These data are required to satisfy suitable axioms, see e.g., \cite[\S 1.1]{dgt2}. We review below some of the consequences which  will be used in what follows. When no confusion arises, we refer to the four-tuple as $V$.

The Fourier coefficients of the fields $Y(\cdot, z)$ endow $V$ with a series of products indexed by $\mathbb{Z}$, that is, $A*_iB:=A_{(i)}B$, for $A,B\in V$. These products are  \textit{weakly} commutative and \textit{weakly} associative.

The \emph{conformal structure} of $V$ realizes the Fourier coefficients $\omega_{(i)}$ as a representation of the Virasoro algebra on $V$ via the identifications $L_p=\omega_{(p+1)}$  and $K = c \cdot \text{id}_V$ for a constant $c \in \CC$ called the \emph{central charge} of $V$. 
Moreover,  $L_0$ is required to act as a degree operator on $V$, i.e., $L_0|_{V_i}=i\cdot \textrm{id}_{V_i}$, and  $L_{-1}$ (the \emph{translation operator}) is given by $L_{-1}A=A_{(-2)}\bm{1}^{V}$, for  $A\in V$.

As a consequence of the axioms, one has $A_{(i)}V_k\subseteq V_{k+\deg(A) -i-1}$ for homogeneous $A\in V$ \cite{zhu}, hence the \emph{degree} of the operator $A_{(i)}$ is defined as $\deg A_{(i)}:= \deg (A) -i-1.$

\subsection{Modules of vertex operator algebras}
\label{sec:VMod} 

There are a number of ways to define a module  over a vertex operator algebra $V$.  We take a $V$-module  $M$ to be a module over the universal enveloping algebra $\mathscr{U}(V)$ of $V$ (defined by I. Frenkel and Zhu \cite{FrenkelZhu}, see also \cite[\S 5.1.5]{bzf}) satisfying three finiteness properties. Namely, we assume that: 
(i) $M$ is a finitely generated $\mathscr{U}(V)$-module;
(ii) $F^0 \,\mathscr{U}(V)v$ is finite-dimensional, for every $v$ in $M$; and 
(iii) for every $v$ in $M$, there exists a positive integer $k$ such that $F^k\, \mathscr{U}(V)v=0$. These conditions are as in  \cite[Def.~2.3.1]{NT}.
Here, $F^k \,\mathscr{U}(V)\subset \mathscr{U}(V)$ is the vector subspace topologically generated by compositions of operators with total degree less than or equal to $-k$.

E. Frenkel and Ben-Zvi  \cite[Thm 5.1.6]{bzf} showed that there is an equivalence of categories between $\mathscr{U}(V)$-modules satisfying property (iii) and the so-called \textit{weak} $V$-modules, which a priori are not graded.   However, with the additional assumptions (i) and (ii), one can show the modules have a grading by the natural numbers.
Such a $V$-module  consists of a pair $\left(M,Y^M(\cdot,z)\right)$, where $M=\oplus_{i\geq 0}\, M_i$ is a $\mathbb{Z}_{\geq 0}$-graded $\mathbb{C}$-vector space,  and \mbox{$Y^M(\cdot,z)\colon V \rightarrow  \textrm{End}(M)\left\llbracket z,z^{-1} \right\rrbracket$}  is a linear function that assigns to $A \in V$ an  $\textrm{End}(M)$-valued vertex operator $Y^M(A,z) :=\sum_{i\in\mathbb{Z}} A^M_{(i)}z^{-i-1}$.   Moreover, by condition (i), if $A\in V$ is  homogeneous, then $ A^M_{(i)} M_k \subseteq M_{k + \deg(A) -i-1}$. 

 The $V$-modules we work with are also known in the literature as finitely generated \textit{admissible} $V$-modules   (see for instance, \cite{AbeBuhlDong} for the definitions of weak and admissible $V$-modules). 

As for $V$, one has that $M$ is also naturally equipped with an action of the Virasoro algebra with central charge $c$, induced by the identification of $\omega_{(p+1)}^M$ with $L_p$. 
When $M$ is a simple $V$-module, there exists  $a_M \in \CC$, called the \emph{conformal dimension} (or \textit{conformal weight}) of $M$, such that $L_0(v)=(\deg(v) + a_M)v$, for every homogeneous  $v$ in $M$ \cite{zhu}. 

The vertex algebra $V$ is a module over itself, sometimes referred to as the \textit{adjoint module} \cite[\S4.1]{lepli} or the \textit{trivial module}. In what follows the set of simple modules over $V$ is denoted $\mathscr{W}$.

\subsection{Contragredient modules}\label{ContraMod} 
Contragredient modules provide a notion of duality for $V$-modules. We recall their definition following \cite[\S 5.2]{fhl}.  
For a vertex algebra $V$ and a $V$-module $\left(M = \oplus_{i\geq 0} M_i, Y^M(-,z)\right)$,  its \textit{contragredient module} is $\left( M',  Y^{M'}(-,z)\right)$, where $M'$ is the graded dual of $M$, that is, $M':=\oplus_{i\geq 0} M_i^\vee$, with $M^\vee_i:=\textrm{Hom}_{\mathbb{C}}(M_i,\mathbb{C})$, and
 $ Y^{M'}(-,z) \colon V \to \mathrm{End} \left(M' \right) \left\llbracket z, z^{-1}\right\rrbracket$
is the unique linear map determined by
\begin{equation*}
 \label{eq:contr} 
 \left\langle Y^{M'}(A,z)\psi, m \right\rangle = \left\langle \psi, Y^M\left(e^{zL_1}(-z^{-2})^{L_0}A, z^{-1}\right)m \right\rangle
\end{equation*}
for $A\in V$, $\psi\in M'$, and $m\in M$. Here  $\langle \cdot, \cdot \rangle$ is the natural pairing between a vector space and its graded dual.

\subsection{The Lie algebra ancillary to $V$}
\label{sec:LV}
The \emph{Lie algebra ancillary to} $V$ is defined as the quotient
\[
\mathfrak{L}(V)= \mathfrak{L}_t(V) := \big( V\otimes \mathbb{C}(\!(t)\!) \big) \big/ \textrm{Im}\, \partial,
\]
where $t$ is a formal variable and $\partial:= L_{-1}\otimes \textrm{id}_{\mathbb{C}(\!(t)\!)} +  \textrm{id}_V \otimes \partial_t$. The image of $A\otimes t^i\in V\otimes \mathbb{C}(\!(t)\!)$ in $\mathfrak{L}(V)$ is denoted by $A_{[i]}$. Observe that $\mathfrak{L}(V)$ is spanned by series of the form $\sum_{i\geq i_0} c_i \, A_{[i]}$, for $A\in V$, $c_i\in\mathbb{C}$, and $i_0\in \mathbb{Z}$. The Lie bracket is induced by
\begin{equation*}
\label{eq:bracket}
\left[A_{[i]}, B_{[j]} \right] := \sum_{k\geq 0} {i \choose k} \left(A_{(k)}\cdot B \right)_{[i+j-k]}.
\end{equation*}
There is a canonical Lie algebra isomorphism  between $\mathfrak{L}(V)$ and the current Lie algebra in \cite{NT}. In what follows, the formal variable $t$ is interpreted as a formal coordinate at a point $P$ on an algebraic curve. A coordinate-free description of $\mathfrak{L}(V)$ is provided in \S \ref{sec:coordfreeLV}.

For a $V$-module $M$, the Lie algebra homomorphism $\mathfrak{L}(V)\rightarrow \textrm{End}(M)$ defined by 
\begin{equation*}
\label{eq:actionLVonM}
\sum_{i\geq i_0} c_i A_{[i]}\mapsto \textrm{Res}_{z=0}\; Y^M(A,z)\sum_{i\geq i_0} c_i z^i dz
\end{equation*}
induces an action of $\mathfrak{L}(V)$ on $M$. For instance, $A_{[i]}$ acts as the Fourier coefficient $A_{(i)}$ of the vertex operator $Y^M(A,z)$.

\subsection{The vertex algebra bundle and the chiral Lie algebra}
\label{sec:coordfreeLV}
Let $(C, P_\bullet)$ be a stable $n$-pointed curve. As illustrated in \cite{bzf} for smooth curves and in \cite{dgt2} for stable curves, one can construct a vector bundle $\mathscr{V}_C$ (the \textit{vertex algebra bundle}) on $C$ whose fiber at each point of $C$ is (non-canonically) isomorphic to $V$. 
For a smooth open subset $U\subset C$ admitting a global coordinate (e.g., if there exists an \'etale map $U\rightarrow \mathbb{A}^1$),   the choice of a global coordinate on $U$ gives a trivialization $\mathscr{V}_C|_U\cong V\times U$.
When $C$ is smooth, the vertex algebra bundle $\mathscr{V}_C$ is constructed via descent along the torsor of formal coordinates at points in $C$. We refer to \cite{dgt2} for the description of $\mathscr{V}_C$ in the nodal case.
The bundle $\mathscr{V}_C$ is naturally equipped with a flat connection $\nabla \colon \mathscr{V}_C \to \mathscr{V}_C \otimes \omega_C$ such that, up to the choice of a formal coordinate $t_i$ at $P_i$, one can identify
\begin{equation}
\label{eq:LVcoordfree}
H^0\left(D_{P_i}^\times, \mathscr{V}_C \otimes \omega_C/ \text{Im}\nabla\right) \cong \mathfrak{L}_{t_i}(V).
\end{equation}
Here $D^\times_{P_i}$  is the punctured formal disk about the marked point $P_i \in C$. 
As shown in \cite[\S\S 19.4.14, 6.6.9]{bzf}, the isomorphism \eqref{eq:LVcoordfree} induces the structure of a Lie algebra independent of coordinates on the left-hand side.

The \emph{chiral Lie algebra} is defined as
\[ 
\mathscr{L}_{C\setminus P_\bullet}(V) := H^0\left(C \setminus P_\bullet, \mathscr{V}_C \otimes \omega_C/ \text{Im}\nabla\right).
\]
This space has indeed  the structure of a Lie algebra after \cite[\S19.4.14]{bzf}.

\subsection{The action of the chiral Lie algebra on $V$-modules}\label{sec:ChiralAction}
Consider the linear map $\varphi$ given by restriction of sections from $C\setminus P_\bullet$ to the $n$ punctured formal disks $D^\times_{P_i}$ using the formal coordinates $t_i$ at $P_i$:
\begin{equation*}
\label{eq:phiL}
\varphi\colon \mathscr{L}_{C\setminus P_\bullet}(V) \rightarrow  \oplus_{i=1}^n \mathfrak{L}_{t_i}(V).
\end{equation*}
After \cite[\S 19.4.14]{bzf}, $\varphi$ is a homomorphism of Lie algebras.
The map $\varphi$ thus induces an action of $\mathscr{L}_{C\setminus P_\bullet}(V)$ on $\mathfrak{L}(V)^{\oplus n}$-modules which is used  to construct coinvariants.
See also \cite[Proposition 3.3.2]{dgt2}.

\subsection{Sheaves of coinvariants and conformal blocks} \label{sec:coinvcblock} 
We briefly recall how to construct  sheaves of coinvariants on $\overline{\mathcal{M}}_{g,n}$, and  refer to \cite[\S 5]{dgt} for a detailed exposition. To a stable $n$-pointed curve $(C, P_\bullet)$ of genus $g$  such that $C\setminus P_\bullet$ is affine, and to $V$-modules $M^1, \dots, M^n$, we associate the space of coinvariants
\[ 
\mathbb{V}(V;M^\bullet)_{(C,P_\bullet)} := M^\bullet_{\mathscr{L}_{C\setminus P_\bullet}(V)}= M^\bullet / \mathscr{L}_{C\setminus P_\bullet}(V) \cdot M^\bullet,
\]
where $M^\bullet= \otimes_{i=1}^n M^i$.
Thanks to the \emph{propagation of vacua}, it is possible to define these spaces also when $C\setminus P_\bullet$ is not affine via a direct limit. Carrying out the construction
relatively over $\overline{\mathcal{M}}_{g,n}$, one defines the quasi-coherent \emph{sheaf of coinvariants}
 $\mathbb{V}_g(V;M^\bullet)$ on $\overline{\mathcal{M}}_{g,n}$ assigned to $M^\bullet$. The dual sheaf $\mathbb{V}_g(V;M^\bullet)^\dagger$ is  the \emph{sheaf of conformal blocks} assigned to  $M^\bullet$.

A brief history of coinvariants and conformal blocks and of the work on their properties  can be found in 
\cite[\S\S 0.1 and 0.2]{dgt2}.


\section{Vector bundles of coinvariants} \label{sec:coinvariants}

Here we review a number of results about vector bundles of coinvariants defined from  representations of vertex algebras satisfying certain natural hypotheses.  Motivated by the new results proved here, we name vertex algebras satisfying such hypotheses as {\em{vertex algebras of CohFT-type}}.

\subsection{Vertex algebras of CohFT-type} \label{sec:CohFT}
We define a vertex algebra $V$ to be \emph{of CohFT-type} if $V$ is a simple, self-contragredient vertex operator algebra such that:

\begin{enumerate}[(I)]
\item \label{CohFTI} $V=\oplus_{i\in \mathbb{Z}_{\ge 0}}V_i$ with $V_0\cong \mathbb{C}$; 
\item \label{CohFTII} $V$ is {\em{rational}}, that is,  every finitely generated $V$-module is a direct sum of simple $V$-modules; and
\item \label{CohFTIII}$V$ is $C_2$-{\em{cofinite}}, that is, the subspace  
\[
C_2(V):=\mathrm{span}_{\mathbb{C}}\left\{A_{(-2)}B \,:\, A, B \in V\right\}
\]
has finite codimension in $V$. 
\end{enumerate}

The set $\mathscr{W}$ of simple modules over a rational vertex algebra is finite, and
a simple module $M=\oplus_{i\geq 0}\,M_i$ over a rational vertex algebra satisfies $\dim M_i<\infty$  \cite{DongLiMasonTwisted}.

The assumptions (I)-(III) on the vertex algebra  have been found in \cite{dgt2} to imply that the sheaves of coinvariants are in fact vector bundles:

\begin{theorem}[{\cite[VB Corollary]{dgt2}}]
\label{thm:dgt2thm2}
For  a vertex algebra $V$ of CohFT-type, the sheaf of coinvariants $\mathbb{V}_g(V;M^\bullet)$ assigned to finitely generated admissible $V$-modules $M^1,\dots,M^n$ is a vector bundle of finite rank on the moduli space $\overline{\mathcal{M}}_{g,n}$.
\end{theorem}

Since $V$ is rational and $C_2$-cofinite, the central charge of $V$ and the conformal dimension of every simple $V$-module are rational numbers \cite{DongLiMasonModular}.
When $V$ is  of CohFT-type, the Chern character of the bundles of coinvariants form a cohomological field theory, as we verify below, hence the name.

\subsection{The connection}
Following \cite{dgt}, the restriction of the vector bundles ${\mathbb{V}}_g(V; M^\bullet)$ to $\mathcal{M}_{g,n}$ support a projectively flat  connection. We can  explicitly describe this  using the language of Atiyah algebras \cite{besh}. Given a line bundle $L$ on a variety, the Atiyah algebra $\mathcal{A}_L$ is the sheaf of first order differential operators acting on $L$. 
An analoguous construction holds for a virtual line bundle $L^{x}$, where $x\in \mathbb{C}$, yielding the Atiyah algebra $x \mathcal{A}_L$ \cite{ts}. With this terminology, the connection on ${\mathbb{V}}_g(V; M^\bullet)$ is explicitly described as follows:

\begin{theorem}[{\cite{dgt}}]
\label{At}
For $n$ simple modules $M^i$ of conformal dimension $a_i$ over a  vertex algebra $V$ of CohFT-type and central charge $c$, the  Atiyah algebra $\frac{c}{2}\mathcal{A}_\Lambda +\sum_{i=1}^n a_i \mathcal{A}_{\Psi_i}$ 
acts on the restriction of ${\mathbb{V}}_g(V; M^\bullet)$ to $\mathcal{M}_{g,n}$, specifying a twisted  $\mathcal{D}$-module structure. 
\end{theorem}

Here $\Lambda$ is the Hodge line bundle on $\mathcal{M}_{g,n}$ and $\Psi_i$ is the cotangent line bundle at the $i$-th marked point on $\mathcal{M}_{g,n}$.
Theorem \ref{At} generalizes  the analogous statement for bundles of coinvariants of integrable representations at a fixed level over affine Lie algebras \cite{ts}. This is proved more generally for \textit{quasi-coherent} sheaves of coinvariants  in \cite[\S 7]{dgt}.

\subsection{Chern classes on ${\mathcal{M}}_{g,n}$}\label{sec:chernMgn}
The explicit description of the connection determines the Chern character of the restriction of $\mathbb{V}_g(V;M^\bullet)$ on $\mathcal{M}_{g,n}$:

\pagebreak

\begin{corollary}[{\cite[Corollary 2]{dgt}}]\label{chMgn}
Let $V$ be a vertex algebra of CohFT-type and central charge~$c$, and let $M^1,\dots,M^n$ be n simple $V$-modules of conformal dimension $a_i$. Then
\[
{\rm ch}\left(  \mathbb{V}_g(V;M^\bullet)  \right) = {\rm rank} \,\mathbb{V}_g(V;M^\bullet) \, \cdot \, \exp\left( \frac{c}{2}\,\lambda + \sum_{i=1}^n a_i \psi_i \right)\, \in H^*\!\left(\mathcal{M}_{g,n}\right).
\]
\end{corollary}

Here $\lambda=c_1(\Lambda)$ and $\psi_i=c_1(\Psi_i)$. The corollary follows from: (a) a vector bundle $E$ over a smooth base $S$ with an action of the Atiyah algebra $\mathcal{A}_{L^{\otimes a}}$, for $L\rightarrow S$ a line bundle and $a\in \mathbb{Q}$, satisfies $c_1(E)=({\rm rank}\, E) \, aL$ \cite[Lemma 5]{mop}; and
(b) the projectively flat connection  implies that ${\rm ch}(E)=({\rm rank}\, E)\,  {\rm exp}(c_1(E)/{\rm rank}\, E)$  \cite[(2.3.3)]{MR909698}.

From Corollary \ref{chMgn}, the total Chern class is
\[
c\left(  \mathbb{V}_g(V; M^\bullet)  \right) =
\left( 1+ \frac{c}{2}\,\lambda + \sum_{i=1}^n a_i \psi_i \right)^{{\rm rank} \,\mathbb{V}_g(V; M^\bullet)} \in H^*\!\left(\mathcal{M}_{g,n}\right).
\]

\subsection{The factorization property}\label{sec:VerlindeSection}
After \cite[Factorization Theorem]{dgt2},  for $V$ of CohFT-type, the bundles $\mathbb{V}_g(V; M^\bullet)$ satisfy the {\em{factorization property}}. Assume that the curve $C$ has one nodal point $Q$, let $\widetilde{C}$ be the partial normalization of $C$ at $Q$, and let $Q_\bullet=(Q_+, Q_-)$ be the pair of preimages of $Q$ in $\widetilde{C}$. 

\begin{theorem}[{\cite[Factorization Theorem]{dgt2}}]\label{thm:Fact} 
Let $V$ be a vertex algebra of CohFT-type. Then
\[
\mathbb{V}\left(V;M^{\bullet}\right)_{(C,P_{\bullet})} \cong \bigoplus_{W \in \mathscr{W}} \mathbb{V}\left(V;M^{\bullet}\otimes W\otimes W' \right)_{\left(\widetilde{C},P_{\bullet}\sqcup Q_\bullet\right)}.
\]
\end{theorem}

When $\widetilde{C}= C_{+}\sqcup C_{-}$ is disconnected, with $Q_{\pm} \in C_{\pm}$, one has:
\[
\mathbb{V}\left(V;M^{\bullet}\otimes W\otimes W' \right)_{(\widetilde{C},P_{\bullet}\sqcup Q_\bullet)}\cong  \mathbb{V}\left(V;M^{\bullet}_{+}\otimes W \right)_{X_{+}}\otimes \mathbb{V}\left(V;M^{\bullet}_{-}\otimes W' \right)_{X_-}
\]
where 
$X_{\pm}=\left(C_{\pm}, P_\bullet|_{C_{\pm}}\sqcup Q_{\pm}\right)$,  and $M^{\bullet}_{\pm}$ are the modules at the $P_\bullet$ on $C_{\pm}$.
The factorization property extends in families of nodal curves. More generally, the \emph{sewing property} holds, extending the factorization property over the formal neighborhood of families of nodal curves \cite{dgt2}.

\subsection{Finding ranks through recursions}
As a consequence of the factorization property, the rank of $\mathbb{V}_g(V; M^\bullet)$, equal to the dimension of the vector space of coinvariants $\mathbb{V}\left(V;M^{\bullet}\right)_{(C,P_{\bullet})}$ at any pointed curve $(C,P_\bullet)$, can be computed when $(C,P_{\bullet})$ is maximally degenerate.  Using propagation of vacua \cite[\S 10.3.1]{bzf} and inserting points with the adjoint module $V$ if necessary, 
one may reduce to the case when all irreducible components of $C$ are rational curves with three special points, and thus  the rank may be expressed as sum of products of dimensions of vector spaces of coinvariants on three-pointed rational curves. 

As the following two examples show, formulas for the ranks can be readily identified in some simple cases. A third such recursive calculation is carried out for bundles defined from modules over even lattice vertex algebras in Example \ref{mLatticeRank}.

\begin{example} The rank of $\mathbb{V}_{1}(V;V)$ on $\overline{\mathcal{M}}_{1,1}$  equals the cardinality of the set of simple $V$-modules $\mathscr{W}$. This follows from factorization and the equality \eqref{DeltaRule} in the next section.
\end{example}

\begin{example}\label{RankOne}Let $V$ be a vertex algebra of CohFT-type with no nontrivial modules.  Bundles of coinvariants of modules  over $V$ on $\overline{\mathcal{M}}_{g,n}$ have rank~$1$.   
\end{example}

In \S \ref{sec:TQFTcalc} we  use the formalism of semisimple TQFTs to reconstruct the ranks from the fusion rules.  Rank computations for bundles defined from modules over even lattice vertex algebras are done from this perspective in Examples \ref{computationlattice} and \ref{mLatticeRankFormal}.


\section{Proof of Theorem \ref{thm:CohFT} and TQFT computations} \label{sec:PrfThmCohFT}
In this section, using the results from \S \ref{sec:coinvariants}, we prove Theorem \ref{thm:CohFT}, showing  bundles of coinvariants defined by representations of vertex algebras of  CohFT-type give rise to cohomological field theories  (for more on CohFTs,  see e.g., \cite[\S 2]{moppz}).   Theorem \ref{thm:CohFT} is proved as in the case of coinvariants from affine Lie algebras treated in \cite[\S 3]{moppz}, and we follow their approach.

\subsection{The CohFT of Chern characters of coinvariants}
\label{sec:CohFTChern}
Let us start by defining the data of the CohFT. Let $V$ be a vertex algebra of CohFT-type.
As $V$ is rational, the set $\mathscr{W}$ of simple $V$-modules is finite.  The Hilbert space of the CohFT is $\mathcal{F}(V):=\mathbb{Q}^\mathscr{W}=\oplus_{W\in \mathscr{W}}\mathbb{Q} \,h_W$.
The CohFT is defined by the classes
\[
\mathrm{ch}\left( \mathbb{V}_g(V; M^\bullet)\right) \in H^*\!\left( \overline{\mathcal{M}}_{g,n}\right),
\]
for $2g-2+n>0$ and $V$-modules $M^1,\dots, M^n$ viewed as elements of $\mathcal{F}(V)$ (hence necessarily finitely-generated), and extending by linearity. 	
The vector space has a pairing $\eta$ defined by $\eta(h_M,h_N)=\delta_{M,N'}$, where $N'$ is the contragredient of $N$. 
In addition, $\mathcal{F}(V)$ has a commutative, associative  product $*$ defined by
\[
h_M*h_N :=  \sum_{W\in\mathscr{W}} \dim \mathbb{V}_0\left(V; M, N, W'\right) h_W,
\]
for $M,N\in\mathscr{W}$ and extending by linearity, with unit $h_V$ corresponding to the adjoint module $V$. 
This product  extends linearly to the \textit{fusion algebra} $\mathcal{F}(V)_{\mathbb{C}}:=\mathcal{F}(V) \otimes_{\mathbb{Q}}\mathbb{C}$. This is a commutative, associative Frobenius algebra with unit. In particular, one has $\eta(a*b,c)=\eta(a,b*c)$, for $a,b,c \in \mathcal{F}(V)_{\mathbb{C}}$.

\begin{proof}[Proof of Theorem \ref{thm:CohFT}]
The axioms necessary to form a CohFT are verified thanks to the factorization property in families \cite[Thm 8.2.2]{dgt2}, propagation of vacua in families (\cite[\S 10.3.1]{bzf} for a single smooth pointed curve, and \cite[Prop.~3.6]{codogni} for families of stable curves, see also \cite[Thm 5.1]{dgt}), and the fact that given simple $V$-modules $M$ and $N$, one has 
\begin{equation}\label{DeltaRule}
{\rm rank}\,  \mathbb{V}_0(V; M,N,V)= \delta_{M,N'}.
\end{equation}
This  follows from the identification of the  three-point, genus zero conformal blocks for simple $V$-modules $M$, $N$, and $W$
with the vector space $\mathrm{Hom}_{A(V)}(A(W)\otimes_{A(V)}M_0,  N_0^\vee)$ \cite{FrenkelZhu, LiFusion}. Here $A(V)$ is Zhu's semisimple associative algebra assigned to $V$, and $A(W)$ is an $A(V)$-bimodule generalizing the algebra $A(V)$  for a $V$-module $W$ \cite{FrenkelZhu}.
For $W=V$, the space of conformal blocks  is thus isomorphic to $\mathrm{Hom}_{A(V)}(M_0, N_0^\vee)$.  The assumption that
$M$ and $N$ are simple $V$-modules gives that $M_0$ and $N_0^\vee$  are simple $A(V)$-modules, and thus \eqref{DeltaRule} follows from Schur's lemma.  We note that statements asserting \eqref{DeltaRule} can be found  in case $V\cong V'$ or for even lattice vertex algebras in \cite{HuangVerlinde2005, fhl, donle}, and \eqref{DeltaRule} is implicitly assumed elsewhere in the literature.

Finally, the CohFT is semisimple, or equivalently, the Frobenius algebra $\mathcal{F}(V)_{\mathbb{C}}$ is semisimple. This follows from the general argument in \cite[\S5 and Prop.~6.1]{BeauvilleVerlinde} using  contragredient duality in genus zero, non-negativity of ranks, and the factorization property. 
Contragredient duality, that is, 
\begin{equation}\label{ContragredientDuality}
\mathrm{rank}\, \mathbb{V}_0 \left(V; M^1, \dots, M^n\right) = \mathrm{rank}\, \mathbb{V}_0 \left(V; (M^1)', \dots, (M^n)'\right),
\end{equation}
is obtained in the affine Lie algebra case as a consequence of the stronger statement \cite[Prop.~2.8]{BeauvilleVerlinde}. In our case, we proceed as follows: By propagation of vacua, we can assume that enough of the modules $M^i$ are equal to $V$, and by the factorization property, we can reduce to compute the rank over a totally degenerate stable rational curve, such that each component has at least one marked point where the adjoint module $V$ is assigned. Then \eqref{ContragredientDuality} follows from \eqref{DeltaRule}.
\end{proof}

\subsection{Computing the TQFT}
\label{sec:TQFTcalc}
As a consequence of Theorem \ref{thm:CohFT}, the ranks of vector bundles of coinvariants of vertex algebras of CohFT-type form a semisimple TQFT. Hence all ranks are determined by the dimensions of vector spaces of coinvariants on a rational curve with three marked points.
We describe here the reconstruction of the TQFT of the ranks from the fusion rules following results on semisimple TQFTs \cite{teleman2012structure, LeeVakil}.

Let $\mathcal{F}(V)_{\mathbb{C}}$ be the fusion ring given from the semisimple TQFT  determined by a vertex algebra $V$ of CohFT-type, as in \S\ref{sec:CohFT}. Let $\{e_i\}_i$ be a \textit{semisimple basis} of $\mathcal{F}(V)$, that is, $\eta(e_i, e_j)=\delta_{i,j}$ and  $e_i * e_j= \delta_{i,j}  \lambda_i e_i$, for some $\lambda_i\in\mathbb{C}$. The values $\lambda_i$ are known as the \textit{semisimple values} of the TQFT. Let $\{e^i\}_i$ be the dual basis to $\{e_i\}_i$. 

\begin{proposition}\label{prop:CohFTRanks} 
The  ranks of the vector bundles of coinvariants on $\overline{\mathcal{M}}_{g,n}$ assigned to $n$ finitely-generated modules over a vertex algebra $V$ of CohFT-type  is given by the following linear functional on 
$\mathcal{F}(V)^{\otimes n}$:
\begin{equation*}
\label{eq:rank}
\sum_i \lambda_i^{2(g-1)+n} \underbrace{e^i\otimes \cdots\otimes e^i}_{n \,\mathrm{times}}.
\end{equation*}
\end{proposition}

\noindent The statement is a special case of \cite[Prop.~4.1]{LeeVakil} applied to our TQFT. Examples are given in \S \ref{sec:ExamplesSection}.
In the case of coinvariants from modules over affine Lie algebras, the above formula reproduces the classical \textit{Verlinde numbers} after some algebraic manipulations (see e.g., \cite{BeauvilleVerlinde,goller2017weighted}).


\section{Chern classes on $\overline{\mathcal{M}}_{g,n}$} \label{sec:chernMgnBar}
In this section we prove Corollary \ref{ChernProp} following \cite{moppz}.
We work  with a vertex algebra $V$ of  CohFT-type. In particular the set $\mathscr{W}$ of simple $V$-modules is finite.
The Chern characters of bundles of coinvariants are given by the polynomial $P_V(a_\bullet)$ defined as a sum over stable graphs. We start by reviewing stable graphs below, and then define the contributions to $P_V(a_\bullet)$ corresponding to vertices, edges, and legs.

\subsection{Stable graphs and module assignments} \label{sec:graph}
A \textit{stable graph} is the dual graph of a stable curve. We only recall  the basic features; for more details see, e.g., \cite{MR3264769}. A stable graph $\Gamma$ comes with a vertex set $V(\Gamma)$, an edge set $E(\Gamma)$, a half-edge set $H(\Gamma)$, and a leg set $L(\Gamma)$. Each leg has a label $i\in\{1,\dots, n\}$, and this gives an isomorphism  $L(\Gamma)\cong\{1,\dots, n\}$. Each edge $e\in E(\Gamma)$ is the union of two half-edges $e=\{h,h'\}$, with $h,h'\in H(\Gamma)$. Each vertex $v\in V(\Gamma)$ has a genus label~$g_v\in \mathbb{Z}_{\geq 0}$ and a valence $n_v$ counting the number of half-edges and legs incident to $v$. The genus of the graph $\Gamma$ is defined as $\sum_{v\in V(\Gamma)} g_v+h^1(\Gamma)$, where $h^1(\Gamma)$ is the first Betti number of $\Gamma$. 

A stable graph $\Gamma$ of genus $g$ with $n$ legs identifies a locally closed stratum in $\overline{\mathcal{M}}_{g, n}$ equal to the image of the glueing map of degree $|{\rm Aut}(\Gamma)|$:
\[
\xi_\Gamma\colon \prod_{v\in V(\Gamma)} \overline{\mathcal{M}}_{g_v, n_v} =:  \overline{\mathcal{M}}_\Gamma \rightarrow  \overline{\mathcal{M}}_{g, n}.
\]

Given a stable graph $\Gamma$, a \textit{module assignment} is a function of type
\[\mu\colon H(\Gamma) \longrightarrow \mathscr{W}\]
such that for $(h,h')\in E(\Gamma)$, one has
$\mu(h')=\mu(h)'$,
that is, $\mu(h')$ is the contragredient module of $\mu(h)$ (\S \ref{ContraMod}).

\subsection{Vertex contributions} \label{sec:vertex}
Fix a stable graph $\Gamma$ and a  module assignment $\mu\colon H(\Gamma)\rightarrow \mathscr{W}$. To each vertex $v$ of $\Gamma$ 
is assigned a collection of simple $V$-modules $M^{i_1}, \dots, M^{i_{n_v}}$, one for each leg or half-edge incident to $v$: for each leg $i$ incident to $v$, the module $M^i$ is assigned to $v$; and for each half-edge $h$ incident to $v$, the module $\mu(h)$ is assigned to $v$. 
The vertex contribution is defined as
\[
{\rm Cont}_\mu (v):= {\rm rank} \,\mathbb{V}_{g_v}\left(V; M^{i_1}, \dots, M^{i_{n_v}}\right).
\]

\subsection{Edge contributions} \label{sec:edge}
Fix a stable graph $\Gamma$ and a  module assignment $\mu\colon H(\Gamma)\rightarrow \mathscr{W}$. For each edge $e=\{h, h'\}$, let $a_{\mu(h)}$ be the conformal dimension of $\mu(h)$. The edge contribution is defined as
\[
{\rm Cont}_\mu (e):= \frac{1 - {\rm e}^{a_{\mu(h)} \left(\psi_h + \psi_{h'}\right)}}{\psi_h + \psi_{h'}}.
\]
This is well defined since $\mu(h)$ and $\mu(h')$ have equal conformal dimension.

\subsection{The polynomial $P_V(a_\bullet)$} \label{sec:poly}
Following the computations of \cite{moppz}, consider the following polynomial with coefficients in $H^*\!\left(\overline{\mathcal{M}}_{g,n} \right)$:
\[
P_V(a_\bullet):=
 {\rm e}^{\frac{c}{2}\lambda}\sum_{\Gamma,\,\mu}\frac{1}{|{\rm Aut}(\Gamma)|}\left(\xi_\Gamma\right)_*
\!\left(\prod_{i=1}^n {\rm e}^{a_i \psi_i} \! \prod_{v\in V(\Gamma)} \!{\rm Cont}_\mu (v) \!\prod_{e\in E(\Gamma)} {\rm Cont}_\mu(e) \!\right)\!,
\]
where $c$ is equal to the central charge of $V$. 
The sum in the formula is over all isomorphism classes of stable graphs $\Gamma$ of genus $g$ with $n$ legs, and over all  module assignments 
$\mu$.
For degree reasons, the exponentials are finite sums, and $P_V(a_\bullet)$ is indeed a polynomial.

For the vector bundle of coinvariants of modules over an affine vertex algebra, \cite{moppz} shows that
${\rm ch}\left(  \mathbb{V}(V; M^\bullet)  \right) = P_V(a_\bullet)$, where $V=L_\ell(\mathfrak{g})$ is the simple affine vertex algebra, and $M^i$ is a simple $L_\ell(\mathfrak{g})$-module of conformal dimension $a_i$, for each $i$.
We extend this result~to prove Corollary \ref{ChernProp}:	

\begin{proof}[Proof of Corollary \ref{ChernProp}]
From \cite[Lemma 2.2]{moppz}, a semisimple CohFT is uniquely determined by the restriction of the classes to $\mathcal{M}_{g,n}$.
More precisely, the results in \cite{moppz} imply that any semisimple CohFT whose restriction to $\mathcal{M}_{g,n}$ is as in Corollary \ref{chMgn} for some $c,a_i\in \mathbb{Q}$ is given by an expression as in the statement.
\end{proof}


\section{Examples and projects}\label{sec:ExamplesSection}

\subsection{Vertex algebras with no nontrivial modules}
We start with coinvariants constructed from a \textit{holomorphic} vertex algebra $V$ of CohFT-type, that is,  a vertex algebra of CohFT-type such that $V$ is the unique simple $V$-module.
  Any bundle of coinvariants of modules  over $V$ on $\overline{\mathcal{M}}_{g,n}$ has rank $1$ and first Chern class equal to $\frac{c}{2}\lambda$, where $c$ is the central charge of $V$. The rank assertion is in Example \ref{RankOne}. The first Chern class follows from Theorem~\ref{ChernProp}, as the conformal dimension of the adjoint  module is zero.

\begin{example}
 There are $71$  holomorphic vertex algebras of CohFT-type with conformal dimension $24$.  This very special class includes the moonshine module vertex algebra $V^{\natural}$ (whose  automorphism group is the monster group), and the vertex algebra given by the Leech lattice \cite{LamShimakura}.
For such   $V=\oplus_{i=0}^{\infty}V_i$, the weight one Lie algebra $V_1$ is
either semi-simple, abelian of rank $24$, or $0$.  If $V_1$ is abelian of rank $24$, then $V$ is isomorphic to the Leech lattice vertex algebra.  If 
$V_1 =0$, it is conjectured that  $V \cong V^{\natural}$.  Vertex algebras with the $69$ other possible Lie algebras $V_1$ have been constructed in \cite{LamShimakura}.  Each gives a vector bundle of coinvariants of rank 1 and first Chern class $12\lambda$.
\end{example}

\subsection{Chern classes of bundles from  even lattice vertex algebras}
\label{Lattice}
As we illustrate below, even lattice vertex  algebras are of CohFT-type, hence following Theorem \ref{thm:dgt2thm2}, their simple modules define vector bundles  on $\overline{\mathcal{M}}_{g,n}$.  

For the definitions of lattice vertex algebras we recommend \cite{Borcherds, flm1, DLattice, lepli}.  We briefly review the notation. 
Let $L$ be a positive-definite even lattice. That is, $L$ is a free abelian group  of finite rank $d$ together with a positive-definite bilinear form $(\cdot, \cdot)$ such that $(\alpha,\alpha)\in 2\mathbb{Z}$ for all $\alpha \in L$. 
One assigns to $L$ the \textit{even lattice vertex algebra} $V_L$. This has finitely many simple modules $\{V_{L+\lambda}\,|\,\lambda\in L'/L\}$, where 
\mbox{$L':=\{\lambda\in L\otimes_{\mathbb{Z}} \mathbb{Q} \, | \, (\lambda,\mu)\in\mathbb{Z}$}, \mbox{for all $\mu\in L\}$} is the dual lattice \cite{DLattice}.
 Contragredient modules are determined by $V_{L+\lambda}'=V_{L-\lambda}$.
The following statement follows from results in the literature.

\begin{proposition}[\cite{Borcherds, FLM, DLattice, donle, DongLiMasonModular}]
\label{prop:latticeva}
For a positive-definite even lattice $L$ of rank $d$, one has:
\begin{enumerate}[i)]

\item The lattice vertex algebra $V_L$ is of CohFT-type with central charge $c=d$;

\item The conformal dimension of the module $V_{L+\lambda}$ is  $\min_{\,\alpha \in L} \frac{(\lambda+\alpha,\lambda+\alpha)}{2}$;

\item For the fusion algebra $\mathcal{F}(V_L)_{\mathbb{C}}=\oplus_{\lambda\in L'/L}\mathbb{C}h_\lambda$, the product is given by
\[
h_{\lambda_1} * h_{\lambda_2} = h_{\lambda_1+\lambda_2}.
\]
\end{enumerate}
\end{proposition} 

\begin{proof}  By \cite{Borcherds, FLM} $V_{L}$ satisfies property $\textrm{(\ref{CohFTI})}$, by  \cite{DLattice} $V_{L}$ is rational, and hence satisfies property $\textrm{(\ref{CohFTII})}$, and by \cite[Proposition 12.5]{DongLiMasonModular}   $V_{L}$ is $C_2$-cofinite, so  satisfies property $\textrm{(\ref{CohFTIII})}$. The central charge is computed in \cite[Theorem 8.10.2]{FLM}, and the conformal dimension is deduced implicitly in \cite[page 260]{DLattice}. The fusion rules are described in \cite[Chapter 12]{donle}.
 \end{proof}

Proposition \ref{prop:latticeva} contains all ingredients needed to compute ranks of bundles of coinvariants from modules over $V_L$ applying Proposition \ref{prop:CohFTRanks} and their Chern characters applying Corollary \ref{ChernProp}. Let us discuss some examples.

\begin{rmk}\label{LatticeTypes} 
There are a number of   lattices $L$ one may use to construct the vertex algebras $V_L$, and it is straightforward to cook up a lattice of almost any rank, whose discriminant group $L'/L$ has arbitrary order. The order of the discriminant is the determinant of the Gram matrix for a basis of the lattice.   For instance, to obtain $L'/L\cong \mathbb{Z}/2k\mathbb{Z}$, for $k \in \mathbb{N}$, one can take a one-dimensional lattice with basis vector $e$ such that $(e,e)=2k$.  

For any root system (see e.g., \cite[pgs 352--355]{Involutions}),  there is a  root lattice $\Lambda$, and for those of type 
$A$, $D$, $E$, $F$, and $G$, the lattice is even, and gives rise to a vertex algebra $V_\Lambda$. Every irreducible root system $\Lambda$ corresponds to a simple Lie algebra $\mathfrak{g}_\Lambda$. If one normalizes the associated bilinear form (encoded by the Dynkin diagram), then in these cases one has $V_{\Lambda} \cong L_{1}(\mathfrak{g}_\Lambda)$, the simple affine vertex algebra at level $1$ (see \cite{FLM} and \cite[Rmk 6.5.8]{lepli} for  details).  The weight lattice gives the dual lattice $\Lambda'$.  For instance, for $\Lambda=A_{m-1}$ the root lattice has rank $m-1$, so that $V_{\Lambda}\cong L_{1}(\mathfrak{s}\mathfrak{l}_m)$ has conformal dimension $m-1$.  

One may also construct vertex algebras $V_L$ by taking $L$ to be the direct sum of  lattices described above, getting quotient lattices $L'/L$ that are of the form $\mathbb{Z}/m_1\mathbb{Z}\oplus \cdots \oplus \mathbb{Z}/m_k\mathbb{Z}$ for arbitrary $m_1, \ldots, m_k$. Such lattices may be interpreted as Mordel-Weil lattices (see e.g., \cite{Shioda, SchShi}). 

The root lattices are very special, and one may easily construct more general lattices with discriminant groups  isomorphic to $\mathbb{Z}/m\mathbb{Z}$.  For instance if $m$ is prime and congruent to $0$ or $3$ mod $4$, this can be done with a rank~$2$ even lattice (the order of the discriminant group for an even lattice of rank~$2$ is always congruent to $0$ or $3$ mod $4$).  As the diversity of quadratic forms cannot be overstated (see e.g., \cite{ConwayS,Conway15, 290}), there are many potentially interesting classes from lattice vertex algebras. 
\end{rmk}

\example Let $V_{L}$ be a vertex algebra given by an even unimodular lattice of rank $d$.  Because the lattice is unimodular, $V_{L}$ is self-contragredient and it has no nontrivial modules.  In particular, any bundle of coinvariants from modules over $V_{L}$  has rank one and first Chern class $\frac{d}{2} \lambda$.

\example \label{computationlattice} 
Consider an even lattice $L$ of rank $d$ with $L'/L \cong \ZZ/2\ZZ$. 
The vertex algebra $V_L$ has two simple modules $V=V_L$ and $W$.
From Proposition \ref{prop:latticeva}, the product in  $\mathcal{F}(V_L)_{\mathbb{C}}=\mathbb{C}h_V\oplus \mathbb{C}h_W$ is given by
\[ 
h_V * h_V= h_V, \qquad h_V * h_W = h_W, \qquad h_W * h_W = h_V.
\] 
With terminology as in \S\ref{sec:TQFTcalc}, a semisimple basis for $\mathcal{F}(V_L)_{\mathbb{C}}$ is
\[ 
e_1:= \dfrac{1}{\sqrt{2}}(h_V+h_W) ,\qquad e_2:= \dfrac{1}{\sqrt{2}}(h_V-h_W),
\]  
with semisimple values both equal to $\sqrt{2}$.
One has $h_V=\dfrac{1}{\sqrt{2}}(e_1+e_2)$ and $h_W = \dfrac{1}{\sqrt{2}}(e_1-e_2)$.
Applying Proposition \ref{prop:CohFTRanks},  the rank of the bundle $\mathbb{V}_g\left(V; V^{\otimes p} \otimes  W^{\otimes q}	\right)$ on $\overline{\mathcal{M}}_{g,n}$ for $p+q=n$ is
\begin{align*} 
\text{rank}\, \mathbb{V}_g\left(V; V^{\otimes p} \otimes  W^{\otimes q}\right) &= \sqrt{2}^{\,2g-2+n}\left(\frac{1}{\sqrt{2}^{\, n}} +(-1)^q \frac{1}{\sqrt{2}^{\, n}}\right)\\
& = 2^{g} \,\delta_{q, \text{ even}}.
\end{align*}
In particular, the rank vanishes when $q$ is odd. Applying Corollary \ref{ChernProp}, when $q=2r$, the Chern character is
\begin{multline*}
\mathrm{ch}\, \mathbb{V}_g\left(V; V^{\otimes p} \otimes  W^{\otimes 2r}\right)= \\
{\rm e}^{\frac{d}{2}\lambda}\sum_{\Gamma}\frac{2^{g-h^1(\Gamma)}}{|{\rm Aut}(\Gamma)|}\left(\xi_\Gamma\right)_*
\!\left(\prod_{i=1}^{2r} {\rm e}^{a \psi_i}  \!\prod_{e\in E(\Gamma)} \frac{1 - {\rm e}^{a \left(\psi_h + \psi_{h'}\right)}}{\psi_h + \psi_{h'}} \!\right).
\end{multline*}
Here $a$ is the conformal dimension of $W$.
The sum in the formula is over those isomorphism classes of stable graphs $\Gamma$ of genus $g$ with $n$ legs such that for each vertex, the number of assigned $W$ at the incident legs is even. Note that the only module assignment $\mu$ contributing nontrivially to the polynomial $P_V$ in this case is $\mu\colon h\mapsto W$, for all half-edges $h$.

\example\label{mLatticeRank} Let $L$ be an even lattice such that $L'/L\cong \mathbb{Z}/m\mathbb{Z}$, for $m\geq 2$. Let $\mathscr{W}=\{V=W_0,  \dots, W_{m-1}\}$ be the set of simple $V_L$-modules. The fusion rules from \cite{donle} give 
\[ 
\text{rank}\,\mathbb{V}_0\left(V_L; W_i \otimes W_j \otimes W_k\right) = \delta_{i+j+k \, \equiv_m \, 0}.
\] 
This implies that 
\[ 
\text{rank}\,\mathbb{V}_0\left(V_L; W_0^{\otimes n_0} \otimes \cdots \otimes W_{m-1}^{\otimes n_{m-1}}\right) = \delta_{\sum_{j=0}^{m-1} j n_j \, \equiv_m \, 0}
\]
and by induction on the genus and the  factorization property, we can further deduce that
\begin{equation}
\label{eq:ranklattivevoazm}
\text{rank}\,\mathbb{V}_g\left(V_L;W_0^{\otimes n_0} \otimes \cdots \otimes W_{m-1}^{\otimes n_{m-1}}\right) = m^g \, \delta_{\sum_{j=0}^{m-1} j n_j \, \equiv_m \, 0}.
\end{equation}

\example\label{mLatticeRankFormal}
One can also obtain the rank found in Example \ref{mLatticeRank} using Proposition \ref{prop:CohFTRanks}.  
Namely, for an even lattice $L$ such that $L'/L\cong \mathbb{Z}/m\mathbb{Z}$, a semisimple basis for the fusion ring $\mathcal{F}(V_L)_{\mathbb{C}}=\oplus_{i=0}^{m-1}\,\mathbb{C}h_i$ is
\[
e_i:= \frac{1}{\sqrt{m}} \sum_{j=0}^{m-1} \rho^{ij} h_j, \qquad \mbox{for $i=0,\dots,m-1$,}
\]
where $\rho\in\mathbb{C}$ is a primitive $m$-th root of unity. One checks $e_i*e_i=\sqrt{m}\,e_i$, hence the semisimple values are all $\sqrt{m}$.
As in Example \ref{computationlattice}, one applies Proposition \ref{prop:CohFTRanks} to recover \eqref{eq:ranklattivevoazm}.

\begin{rmk}
Examples \ref{computationlattice}--\ref{mLatticeRankFormal} show that  ranks of bundles of coinvariants from modules over  an even lattice $L$ such that $L'/L\cong \mathbb{Z}/m\mathbb{Z}$ coincide with ranks of bundles of coinvariants from modules over the affine Lie algebra $\widehat{\mathfrak{sl}}_m$ at level one. 
However, while the ranks depend only on $L'/L$,  the Chern characters depend additionally on the quadratic form of $L$, responsible for the conformal dimension of the irreducible $V_L$-modules. 
It is reasonable to expect that classes from lattice vertex algebras could give a  collection of CohFTs larger than the one  obtained from the affine Lie algebra case.
\end{rmk}

\subsection{Commutants and the parafermion vertex algebras}\label{para}

For a vertex algebra $V$ and a vertex subalgebra $U$ of $V$, one may construct the \textit{commutant}, or \textit{coset}, vertex algebra  $\mathrm{Com}_V(U)$ of $U$ in $V$  \cite{FrenkelZhu}.  It  would be interesting to study the Chern classes for bundles of coinvariants of modules over $\mathrm{Com}_V(U)$ for pairs $U \subset V$ such that $\mathrm{Com}_V(U)$ is of CohFT-type.

Conjecturally, if $U$ and $V$ are both of CohFT-type, then $\mathrm{Com}_V(U)$ is also of CohFT-type.  However, $U$ and $V$ need not be of CohFT-type:  one such example is given by the  well-studied family of cosets of the Heisenberg vertex  algebra  in the affine vertex  algebra $L_{k}(\mathfrak{g})$ for a finite-dimensional simple Lie algebra $\mathfrak{g}$ at level $k\in \mathbb{Z}$.  The Heisenberg vertex algebra is not rational, nor $C_2$-cofinite (see e.g., \cite{bzf} for a discussion of the Heisenberg vertex  algebra). 
Nevertheless, the \textit{parafermion vertex algebras}  are known to be of CohFT-type  \cite{DongRen} (see also \cite{ArakawaLamYamada,DongLamWangYamada, DongLamYamada, DongWang1, DongWang2, DongWang3}). 
These are related to $W$-algebras  \cite{ArakawaLamYamada}. 
The necessary invariants for expressing the Chern classes of bundles of coinvariants of modules over parafermion vertex algebras are known from \cite{DongKacLi, DongRen, DWFusion}, and one could proceed as in~\S\ref{Lattice}.

\subsection{Orbifold vertex algebras}
    Let $G$ be a subgroup of the group of automorphisms of a vertex algebra $V$.  The \textit{orbifold vertex  algebra} $V^{G}$ consists of the fixed points of $G$ in $V$. 
In case $V$ is of CohFT-type, its full group of automorphisms $G$ is a finite-dimensional algebraic group \cite{DoGr}.  If  $G$ is also solvable, then $V^{G}$ will also be of CohFT-type  \cite{MiyamotoCyclic, CarnahanMiyamoto}. Conjecturally, $V^G$ is always of CohFT-type.
 We note that  even if  $V$  is holomorphic, and therefore has no non-trivial modules, the vertex algebra  $V^G$ will not generally be holomorphic \cite{gk, DVVV, DPR}.  
 
One could for instance consider orbifold vertex algebras created from  parafermion vertex algebras.
In some  cases, the simple modules and the fusion rules are known in the literature \cite{JW2017, JW2019, JW2019}. 

Similarly, one could construct orbifold vertex algebras starting from lattice vertex algebras.  
 In \cite{BE}, simple modules for orbifolds $V^{G}_L$, where $G$ is generated by an isometry of order two, are classified and their fusion rules are given.  Explicit examples with root lattices and Dynkin diagram automorphisms are given in \cite[\S 4]{BE}.  
 

\section{Questions}\label{Necessary}

Summarizing, bundles of coinvariants defined by  modules over vertex algebras of CohFT-type share three important properties  with their classical counterparts: They (i)  support a projectively flat logarithmic connection \cite{tuy,dgt}; (ii)  satisfy the {\em{factorization property}}, a reflection of their  underlying combinatorial structure \cite{tuy,dgt2}; and (iii) give rise to cohomological field theories, as we show here.  

As described in Remark \ref{LatticeTypes}, bundles of coinvariants from lattice vertex algebras are generalizations of those given by affine Lie algebras at level one.   It is natural to expect that other known properties of the classical case extend to the vertex algebra case, and a number of questions come to mind.  

\subsubsection*{Question 1}
Given a simple, simply connected algebraic group $G$ with associated Lie algebra $\mathrm{Lie}(G)=\mathfrak{g}$ and $\ell \in \mathbb{Z}_{>0}$,
the simple vertex algebra $V=L_{\ell}(\mathfrak{g})$ is of CohFT-type \cite{FrenkelZhu,DongLiMasonModular}.   By \cite{bl1, Faltings, KNR}, for a smooth algebraic curve $C$, 
there is a natural line  bundle $D$  on  the moduli stack $\mathrm{Bun}_{G}(C)$ of $G$-bundles on $C$ such that, for any point $P$ in  $C$, there is a canonical isomorphism 
\[\mathbb{V}\left(L_{\ell}(\mathfrak{g}); L_{\ell}(\mathfrak{g})\right)_{(C,P)}^{\dagger} 
\cong H^{0}\left(\textrm{Bun}_{G}(C), D^\ell\right),\]  where $L_{\ell}(\mathfrak{g})$  is the adjoint module over itself. 
By  \cite{P, LS}, given $V$-modules $M^{\bullet}$,
there is a line bundle  $L$ 
    on the moduli stack of quasi-parabolic $G$-bundles $\textrm{ParBun}_{G}(C, P_{\bullet})$, for which
$\mathbb{V}(V; M^{\bullet})_{(C, P_{\bullet})}^{\dagger}$ is isomorphic to the global sections of
$L$ on $\textrm{ParBun}_{G}(C, P_{\bullet})$.  This geometric picture holds for stable curves with singularities as well \cite{BF1} (see also \cite{BGDet}).

The automorphism group $\mathrm{Aut}(V)$ of a vertex algebra $V$ of CohFT-type is a finite-dimensional algebraic group \cite{DoGr}.  
The connected component  $\mathrm{Aut}(V)^0$ of $\mathrm{Aut}(V)$ containing the identity has been  described in a number of cases  \cite{DGR99, DGLattice, DGR2, DoGr}. 
For instance for $V=L_{\ell}(\mathfrak{g})$, one has $\mathrm{Lie}\left(\mathrm{Aut}(V)^0\right)\cong V_1 \cong \mathfrak{g}$ \cite{DoGr}.  For $V$ of CohFT-type, can one find  a geometric realization for conformal blocks defined from modules over $V$, for instance involving algebraic structures on curves related to $\mathrm{Aut}(V)^0$?  Ideas in this direction  have been considered  in \cite{UenoTheta} and \cite{BZFGeometric}.

\subsubsection*{Question 2}   Gromov-Witten invariants of smooth projective homogeneous spaces define base-point-free classes on $\overline{\mathcal{M}}_{0,n}$; divisors defined from Gromov-Witten invariants of $\mathbb{P}^r=\mathbb{G}r(1,r+1)$ are equivalent to first Chern classes of bundles   from integrable modules at level one over $\mathfrak{sl}_{r+1}$ \cite[Props 1.4 and 3.1]{BG2}.   Numerical evidence suggests a more general connection between classes of bundles 
at level $\ell$ with Gromov-Witten divisors for Grassmannians $\mathbb{G}r(\ell,r+\ell)$ \cite{BG2}.  By Witten's Dictionary, the quantum cohomology of Grassmannians can be used to compute ranks of conformal blocks bundles in type $A$ for any level \cite{WittenDictionary}.     Are there connections between other Gromov-Witten theories and the more general bundles of coinvariants studied here?

\subsubsection*{Question 3} Vector bundles defined by representations of  affine Lie algebras are globally generated in genus zero, and so Chern classes have valuable positivity properties.  For instance, first Chern classes are base-point-free,  giving rise to morphisms  \cite{fakhr}.   
 Can one give sufficient conditions on vertex algebras of CohFT-type and their modules so that the vector bundles of coinvariants are globally generated?  Chern classes of bundles  from certain Virasoro vertex algebras are not nef, so further assumptions must be made.
See \cite{dg} for initial results along these lines.

\subsubsection*{Question 4}   Bundles of coinvariants from affine Lie algebras
 give rise to morphisms from $\overline{\mathcal{M}}_{0,n}$ to Grassmannian varieties  \cite{fakhr}.
In the special case where the Lie algebra is of type A and modules are at level one, we know the image varieties parametrize configurations of weighted points on rational normal curves in projective spaces \cite{GiansiracusaSimpson, Giansiracusa, gg}.    If Chern classes given by representations of particular types of vertex algebras are base-point-free, can one give modular interpretations for the images of their associated maps?

\subsubsection*{Question 5}
Classes  associated to  bundles of coinvariants on $\overline{\mathcal{M}}_{0,n}$ defined by $V=L_{\ell}(\mathfrak{g})$ satisfy scaling and level-rank identities, and are zero above a critical level, allowing one to give sufficient conditions for when they lie on extremal faces of the nef cone \cite{ags, BGM, BGMAdditive}.   Do Chern classes studied here satisfy similar identities? Can one find criteria to ensure they  lie on extremal faces of cones of nef cycles?  Do bundles defined by particular modules over vertex operator algebras  generate extremal rays?  Bundles  from vertex operator algebras constructed from exotic lattices may be relevant.

\subsubsection*{Question 6}In \cite[Theorem 5.1]{BologneseGiansiracusa}, using factorization, Verlinde bundles
constructed from level one  integrable modules over $\mathfrak{sl}_{r+1}$ are shown to be isomorphic to both  GIT bundles \cite{BologneseGiansiracusa}, and to the $r$-th tensor 
power of cyclic bundles studied in \cite{FedCyclic}.  Are there other such identifications, for instance involving the line bundles of coinvariants on $\overline{\mathcal{M}}_{0,n}$  constructed from even lattice theories, as discussed in Example \ref{mLatticeRank}?


\section*{Acknowledgements} 

The authors are grateful to Yi-Zhi Huang for helpful discussions, and to Marian, Oprea, Pandharipande, Pixton, and Zvonkine, for their work in \cite{moppz}.
Conversations with Aaron Bertram from some years ago helped with TQFTs computations.  The authors also thank Daniel Krashen and Bin~Gui for discussions about lattices. We thank the referee for their valuable comments. Gibney was supported by NSF DMS--1902237.



\bibliographystyle{alphanumN}
\bibliography{Biblio}

\newcommand{\etalchar}[1]{$^{#1}$}
\begin{thebibliography}{MOP{\etalchar{+}}2}

\bibitem[ABD]{AbeBuhlDong}
T. Abe, G. Buhl, and C. Dong.
\newblock Rationality, regularity, and {$C_2$}-cofiniteness.
\newblock {\em Trans. Amer. Math. Soc.}, 356(8):3391--3402, 2004.

\bibitem[ADJR]{DWFusion}
C. Ai, C. Dong, X. Jiao, and L. Ren.
\newblock The irreducible modules and fusion rules for the parafermion vertex
  operator algebras.
\newblock {\em Trans. Amer. Math. Soc.}, 370(8):5963--5981, 2018.

\bibitem[AGS]{ags}
V. Alexeev, A. Gibney, and D. Swinarski.
\newblock Higher-level {$\mathfrak{sl}_2$} conformal blocks divisors on
  {$\overline{\mathcal{M}}_{0,n}$}.
\newblock {\em Proc. Edinb. Math. Soc. (2)}, 57(1):7--30, 2014.

\bibitem[ALY]{ArakawaLamYamada}
T. Arakawa, C.~H. Lam, and H. Yamada.
\newblock Zhu's algebra, {$C_2$}-algebra and {$C_2$}-cofiniteness of
  parafermion vertex operator algebras.
\newblock {\em Adv. Math.}, 264:261--295, 2014.

\bibitem[BE]{BE}
B. Bakalov and J. Elsinger.
\newblock Orbifolds of lattice vertex algebras under an isometry of order two.
\newblock {\em J. Algebra}, 441:57--83, 2015.

\bibitem[Bea]{BeauvilleVerlinde}
A. Beauville.
\newblock Conformal blocks, fusion rules and the {V}erlinde formula.
\newblock In {\em Proceedings of the {H}irzebruch 65 {C}onference on
  {A}lgebraic {G}eometry ({R}amat {G}an, 1993)}, volume~9 of {\em Israel Math.
  Conf. Proc.}, pages 75--96. Bar-Ilan Univ., Ramat Gan, 1996.

\bibitem[Bel]{WittenDictionary}
P. Belkale.
\newblock The tangent space to an enumerative problem.
\newblock In {\em Proceedings of the {I}nternational {C}ongress of
  {M}athematicians. {V}olume {II}}, pages 405--426. Hindustan Book Agency, New
  Delhi, 2010.

\bibitem[BF]{BF1}
P. Belkale and N. Fakhruddin.
\newblock Triviality properties of principal bundles on singular curves.
\newblock {\em Algebr. Geom.}, 6(2):234--259, 2019.

\bibitem[BG1]{BGDet}
P. Belkale and A. Gibney.
\newblock On finite generation of the section ring of the determinant of
  cohomology line bundle.
\newblock {\em Trans. Amer. Math. Soc.}, 371(10):7199--7242, 2019.

\bibitem[BG2]{BG2}
P. Belkale and A. Gibney.
\newblock {Basepoint Free Cycles on $\overline{\operatorname{M}}_{0,n}$ from
  Gromov-Witten Theory}.
\newblock {\em International Mathematics Research Notices}, 09 2019.
\newblock rnz184.

\bibitem[BG3]{BologneseGiansiracusa}
M. Bolognesi and N. Giansiracusa.
\newblock Factorization of point configurations, cyclic covers, and conformal
  blocks.
\newblock {\em J. Eur. Math. Soc. (JEMS)}, 17(10):2453--2471, 2015.

\bibitem[BGM1]{BGMAdditive}
P. Belkale, A. Gibney, and S. Mukhopadhyay.
\newblock Nonvanishing of conformal blocks divisors on {$\overline M_{0,n}$}.
\newblock {\em Transform. Groups}, 21(2):329--353, 2016.

\bibitem[BGM2]{BGM}
P. Belkale, A. Gibney, and S. Mukhopadhyay.
\newblock Vanishing and identities of conformal blocks divisors.
\newblock {\em Algebr. Geom.}, 2(1):62--90, 2015.

\bibitem[BH]{290}
M. Bhargava and J. Hanke.
\newblock Universal quadratic forms and the 290-theorem.
\newblock 2017.

\bibitem[BL]{bl1}
A. Beauville and Y. Laszlo.
\newblock Conformal blocks and generalized theta functions.
\newblock {\em Comm. Math. Phys.}, 164(2):385--419, 1994.

\bibitem[Bor]{Borcherds}
R.~E. Borcherds.
\newblock Vertex algebras, {K}ac-{M}oody algebras, and the {M}onster.
\newblock {\em Proc. Nat. Acad. Sci. U.S.A.}, 83(10):3068--3071, 1986.

\bibitem[BS]{besh}
A.~A. Beilinson and V.~V. Schechtman.
\newblock Determinant bundles and {V}irasoro algebras.
\newblock {\em Comm. Math. Phys.}, 118(4):651--701, 1988.

\bibitem[BZF]{BZFGeometric}
D. Ben-Zvi and E. Frenkel.
\newblock Geometric realization of the {S}egal-{S}ugawara construction.
\newblock In {\em Topology, geometry and quantum field theory}, volume 308 of
  {\em London Math. Soc. Lecture Note Ser.}, pages 46--97. Cambridge Univ.
  Press, Cambridge, 2004.

\bibitem[CM]{CarnahanMiyamoto}
S. Carnahan and M. Miyamoto.
\newblock Regularity of fixed-point vertex operator subalgebras.
\newblock {\em arXiv preprint arXiv:1603.05645}, 2016.

\bibitem[Cod]{codogni}
G. Codogni.
\newblock Vertex algebras and {T}eichm\"uller modular forms.
\newblock {\em arXiv:1901.03079}, pages 1--31, 2019.

\bibitem[Con1]{Conway15}
J.~H. Conway.
\newblock Universal quadratic forms and the fifteen theorem.
\newblock In {\em Quadratic forms and their applications ({D}ublin, 1999)},
  volume 272 of {\em Contemp. Math.}, pages 23--26. Amer. Math. Soc.,
  Providence, RI, 2000.

\bibitem[Con2]{ConwayS}
J.~H. Conway.
\newblock {\em The sensual (quadratic) form}, volume~26 of {\em Carus
  Mathematical Monographs}.
\newblock Mathematical Association of America, Washington, DC, 1997.
\newblock With the assistance of Francis Y. C. Fung.

\bibitem[DG1]{dg}
C. Damiolini and A. Gibney.
\newblock On global generation of vector bundles on the moduli space of curves
  from representations of vertex operator algebras.
\newblock {\em arXiv preprint arXiv:2107.06923}, 2021.

\bibitem[DG2]{DGLattice}
C. Dong and R.~L. Griess, Jr.
\newblock Rank one lattice type vertex operator algebras and their automorphism
  groups.
\newblock {\em J. Algebra}, 208(1):262--275, 1998.

\bibitem[DG3]{DoGr}
C. Dong and R.~L. Griess, Jr.
\newblock Automorphism groups and derivation algebras of finitely generated
  vertex operator algebras.
\newblock {\em Michigan Math. J.}, 50(2):227--239, 2002.

\bibitem[DG4]{DGR2}
C. Dong and R.~L. Griess, Jr.
\newblock The rank-2 lattice-type vertex operator algebras {$V_L^+$} and their
  automorphism groups.
\newblock {\em Michigan Math. J.}, 53(3):691--715, 2005.

\bibitem[DGR]{DGR99}
C. Dong, R.~L. Griess, Jr., and A. Ryba.
\newblock Rank one lattice type vertex operator algebras and their automorphism
  groups. {II}. {$E$}-series.
\newblock {\em J. Algebra}, 217(2):701--710, 1999.

\bibitem[DGT1]{dgt2}
C. Damiolini, A. Gibney, and N. Tarasca.
\newblock On factorization and vector bundles of conformal blocks from vertex
  algebras.
\newblock {\em Submitted, arXiv:1909.04683}, 2019.

\bibitem[DGT2]{dgt}
C. Damiolini, A. Gibney, and N. Tarasca.
\newblock Conformal blocks from vertex algebras and their connections on
  {$\overline{\mathcal{M}}_{g, n}$}.
\newblock {\em Geom. Topol.}, 25(5):2235--2286, 2021.

\bibitem[DKR]{DongKacLi}
C. Dong, V. Kac, and L. Ren.
\newblock Trace functions of the parafermion vertex operator algebras.
\newblock {\em Adv. Math.}, 348:1--17, 2019.

\bibitem[DL]{donle}
C. Dong and J. Lepowsky.
\newblock {\em Generalized vertex algebras and relative vertex operators},
  volume 112 of {\em Progress in Maths}.
\newblock Birkh\"{a}user Boston, Inc., Boston, MA, 1993.

\bibitem[DLM1]{DongLiMasonTwisted}
C. Dong, H. Li, and G. Mason.
\newblock Twisted representations of vertex operator algebras.
\newblock {\em Math. Ann.}, 310(3):571--600, 1998.

\bibitem[DLM2]{DongLiMasonModular}
C. Dong, H. Li, and G. Mason.
\newblock Modular-invariance of trace functions in orbifold theory and
  generalized {M}oonshine.
\newblock {\em Comm. Math. Phys.}, 214(1):1--56, 2000.

\bibitem[DLWY]{DongLamWangYamada}
C. Dong, C.~H. Lam, Q. Wang, and H. Yamada.
\newblock The structure of parafermion vertex operator algebras.
\newblock {\em J. Algebra}, 323(2):371--381, 2010.

\bibitem[DLY]{DongLamYamada}
C. Dong, C.~H. Lam, and H. Yamada.
\newblock {$W$}-algebras related to parafermion algebras.
\newblock {\em J. Algebra}, 322(7):2366--2403, 2009.

\bibitem[Don]{DLattice}
C. Dong.
\newblock Vertex algebras associated with even lattices.
\newblock {\em J. Algebra}, 161(1):245--265, 1993.

\bibitem[DPR]{DPR}
R. Dijkgraaf, V. Pasquier, and P. Roche.
\newblock Quasi {H}opf algebras, group cohomology and orbifold models.
\newblock volume 18B, pages 60--72 (1991). 1990.
\newblock Recent advances in field theory (Annecy-le-Vieux, 1990).

\bibitem[DR]{DongRen}
C. Dong and L. Ren.
\newblock Representations of the parafermion vertex operator algebras.
\newblock {\em Adv. Math.}, 315:88--101, 2017.

\bibitem[DVVV]{DVVV}
R. Dijkgraaf, C. Vafa, E. Verlinde, and H. Verlinde.
\newblock The operator algebra of orbifold models.
\newblock {\em Comm. Math. Phys.}, 123(3):485--526, 1989.

\bibitem[DW1]{DongWang1}
C. Dong and Q. Wang.
\newblock The structure of parafermion vertex operator algebras: general case.
\newblock {\em Comm. Math. Phys.}, 299(3):783--792, 2010.

\bibitem[DW2]{DongWang2}
C. Dong and Q. Wang.
\newblock On {$C_2$}-cofiniteness of parafermion vertex operator algebras.
\newblock {\em J. Algebra}, 328:420--431, 2011.

\bibitem[DW3]{DongWang3}
C. Dong and Q. Wang.
\newblock Parafermion vertex operator algebras.
\newblock {\em Front. Math. China}, 6(4):567--579, 2011.

\bibitem[Fak]{fakhr}
N. Fakhruddin.
\newblock Chern classes of conformal blocks.
\newblock In {\em Compact moduli spaces and vector bundles}, volume 564 of {\em
  Contemp. Math.}, pages 145--176. Amer. Math. Soc., Providence, RI, 2012.

\bibitem[Fal]{Faltings}
G. Faltings.
\newblock A proof for the {V}erlinde formula.
\newblock {\em J. Algebraic Geom.}, 3(2):347--374, 1994.

\bibitem[FBZ]{bzf}
E. Frenkel and D. Ben-Zvi.
\newblock {\em Vertex algebras and algebraic curves}, volume~88 of {\em
  Mathematical Surveys and Monographs}.
\newblock Am Mathem Soc, Providence, RI, second edition, 2004.

\bibitem[Fed]{FedCyclic}
M. Fedorchuk.
\newblock Cyclic covering morphisms on {$\overline{\mathcal{M}}_{0,n}$}.
\newblock {\em arXiv preprint arXiv:1105.0655}, 2011.

\bibitem[FHL]{fhl}
I.~B. Frenkel, Y.-Z. Huang, and J. Lepowsky.
\newblock On axiomatic approaches to vertex operator algebras and modules.
\newblock {\em Mem. Amer. Math. Soc.}, 104(494):viii+64, 1993.

\bibitem[FLM1]{FLM}
I. Frenkel, J. Lepowsky, and A. Meurman.
\newblock {\em Vertex operator algebras and the {M}onster}, volume 134 of {\em
  Pure and Applied Mathematics}.
\newblock Academic Press, Inc., Boston, MA, 1988.

\bibitem[FLM2]{flm1}
I.~B. Frenkel, J. Lepowsky, and A. Meurman.
\newblock A natural representation of the {F}ischer-{G}riess {M}onster with the
  modular function {J} as character.
\newblock {\em Proceedings of the National Academy of Sciences},
  81(10):3256--3260, 1984.

\bibitem[FZ]{FrenkelZhu}
I.~B. Frenkel and Y. Zhu.
\newblock Vertex operator algebras associated to representations of affine and
  {V}irasoro algebras.
\newblock {\em Duke Math. J.}, 66(1):123--168, 1992.

\bibitem[GG]{gg}
N. Giansiracusa and A. Gibney.
\newblock The cone of type {$A$}, level 1, conformal blocks divisors.
\newblock {\em Adv. Math.}, 231(2):798--814, 2012.

\bibitem[Gia]{Giansiracusa}
N. Giansiracusa.
\newblock Conformal blocks and rational normal curves.
\newblock {\em J. Algebraic Geom.}, 22(4):773--793, 2013.

\bibitem[Giv1]{givental2001gromov}
A.~B. Givental.
\newblock Gromov--{Witten} invariants and quantization of quadratic
  {Hamiltonians}.
\newblock {\em Moscow Mathematical Journal}, 1(4):551--568, 2001.

\bibitem[Giv2]{givental2001semisimple}
A.~B. Givental.
\newblock Semisimple {Frobenius} structures at higher genus.
\newblock {\em International mathematics research notices},
  2001(23):1265--1286, 2001.

\bibitem[GK]{gk}
T. {Gem\"unden} and C. {Keller}.
\newblock {Orbifolds of lattice vertex operator algebras at $d=48$ and $d=72$}.
\newblock {\em arXiv e-prints arXiv:1802.10581}, Feb 2018.

\bibitem[Gol]{goller2017weighted}
T. Goller.
\newblock A weighted topological quantum field theory for quot schemes on
  curves.
\newblock {\em Mathematische Zeitschrift}, pages 1--36, 2017.

\bibitem[GS]{GiansiracusaSimpson}
N. Giansiracusa and M. Simpson.
\newblock {GIT} compactifications of {$\mathcal{M}_{0,n}$} from conics.
\newblock {\em Int. Math. Res. Not. IMRN}, (14):3315--3334, 2011.

\bibitem[Hua]{HuangVerlinde2005}
Y.-Z. Huang.
\newblock Vertex operator algebras, the {V}erlinde conjecture, and modular
  tensor categories.
\newblock {\em Proc. Natl. Acad. Sci. USA}, 102(15):5352--5356, 2005.

\bibitem[JW1]{JW2017}
C. Jiang and Q. Wang.
\newblock Representations of $\mathbb{Z}_2$-orbifold of the parafermion vertex
  operator algebra $k(\mathfrak{s}\mathfrak{l}_2,k)$.
\newblock {\em arXiv:1904.01798v1}, 2017.

\bibitem[JW2]{JW2019}
C. Jiang and Q. Wang.
\newblock Fusion rules for $\mathbb{Z}_2$-orbifolds of affine and parafermion
  vertex operator algebras.
\newblock {\em arXiv:1904.01798v1}, 2019.

\bibitem[KMRT]{Involutions}
M.-A. Knus, A. Merkurjev, M. Rost, and J.-P. Tignol.
\newblock {\em The book of involutions}, volume~44 of {\em American
  Mathematical Society Colloquium Publications}.
\newblock American Mathematical Society, Providence, RI, 1998.
\newblock With a preface in French by J. Tits.

\bibitem[KNR]{KNR}
S. Kumar, M.~S. Narasimhan, and A. Ramanathan.
\newblock Infinite {G}rassmannians and moduli spaces of {$G$}-bundles.
\newblock {\em Math. Ann.}, 300(1):41--75, 1994.

\bibitem[Kob]{MR909698}
S. Kobayashi.
\newblock {\em Differential geometry of complex vector bundles}, volume~15 of
  {\em Publications of the Mathematical Society of Japan}.
\newblock Princeton University Press, Princeton, NJ, 1987.
\newblock Kan\^o Memorial Lectures, 5.

\bibitem[Li]{LiFusion}
H. Li.
\newblock Determining fusion rules by {$A(V)$}-modules and bimodules.
\newblock {\em J. Algebra}, 212(2):515--556, 1999.

\bibitem[LL]{lepli}
J. Lepowsky and H. Li.
\newblock {\em Introduction to vertex operator algebras and their
  representations}, volume 227 of {\em Progress in Maths}.
\newblock Birkh\"{a}user Boston, Inc., Boston, MA, 2004.

\bibitem[LS1]{LamShimakura}
C.~H. Lam and H. Shimakura.
\newblock 71 holomorphic vertex operator algebras of central charge 24.
\newblock {\em Bull. Inst. Math. Acad. Sin. (N.S.)}, 14(1):87--118, 2019.

\bibitem[LS2]{LS}
Y. Laszlo and C. Sorger.
\newblock The line bundles on the moduli of parabolic {$G$}-bundles over curves
  and their sections.
\newblock {\em Ann. Sci. \'{E}cole Norm. Sup. (4)}, 30(4):499--525, 1997.

\bibitem[LV]{LeeVakil}
Y.-P. Lee and R. Vakil.
\newblock Algebraic structures on the topology of moduli spaces of curves and
  maps.
\newblock In {\em Surveys in differential geometry. {V}ol. {XIV}. {G}eometry of
  {R}iemann surfaces and their moduli spaces}, volume~14 of {\em Surv. Differ.
  Geom.}, pages 197--219. Int. Press, Somerville, MA, 2009.

\bibitem[Miy]{MiyamotoCyclic}
M. Miyamoto.
\newblock {$C_2$}-cofiniteness of cyclic-orbifold models.
\newblock {\em Comm. Math. Phys.}, 335(3):1279--1286, 2015.

\bibitem[MOP1]{mop}
A. Marian, D. Oprea, and R. Pandharipande.
\newblock The first {C}hern class of the {V}erlinde bundles.
\newblock In {\em String-{M}ath 2012}, volume~90 of {\em Proc. Sympos. Pure
  Math.}, pages 87--111. Amer. Math. Soc., Providence, RI, 2015.

\bibitem[MOP{\etalchar{+}}2]{moppz}
A. Marian, D. Oprea, R. Pandharipande, A. Pixton, and D. Zvonkine.
\newblock The {C}hern character of the {V}erlinde bundle over
  {$\overline{\mathcal{M}}_{g,n}$}.
\newblock {\em J. Reine Angew. Math.}, 732:147--163, 2017.

\bibitem[NT]{NT}
K. Nagatomo and A. Tsuchiya.
\newblock Conformal field theories associated to regular chiral vertex operator
  algebras. {I}. {T}heories over the projective line.
\newblock {\em Duke Math. J.}, 128(3):393--471, 2005.

\bibitem[Pan]{pandharipande2017cohomological}
R. Pandharipande.
\newblock Cohomological field theory calculations.
\newblock {\em arXiv preprint arXiv:1712.02528}, 2017.

\bibitem[Pau]{P}
C. Pauly.
\newblock Espaces de modules de fibr\'{e}s paraboliques et blocs conformes.
\newblock {\em Duke Math. J.}, 84(1):217--235, 1996.

\bibitem[PPZ]{MR3264769}
R. Pandharipande, A. Pixton, and D. Zvonkine.
\newblock Relations on {$\overline{\mathcal{M}}_{g,n}$} via {$3$}-spin
  structures.
\newblock {\em J. Amer. Math. Soc.}, 28(1):279--309, 2015.

\bibitem[Shi]{Shioda}
T. Shioda.
\newblock Mordell-{W}eil lattices for higher genus fibration over a curve.
\newblock In {\em New trends in algebraic geometry ({W}arwick, 1996)}, volume
  264 of {\em London Math. Soc. Lecture Note Ser.}, pages 359--373. Cambridge
  Univ. Press, Cambridge, 1999.

\bibitem[SS]{SchShi}
M. Sch\"{u}tt and T. Shioda.
\newblock Elliptic surfaces.
\newblock In {\em Algebraic geometry in {E}ast {A}sia---{S}eoul 2008},
  volume~60 of {\em Adv. Stud. Pure Math.}, pages 51--160. Math. Soc. Japan,
  Tokyo, 2010.

\bibitem[Tel]{teleman2012structure}
C. Teleman.
\newblock The structure of {2D} semi-simple field theories.
\newblock {\em Inventiones mathematicae}, 188(3):525--588, 2012.

\bibitem[TK]{TK}
A. Tsuchiya and Y. Kanie.
\newblock Vertex operators in the conformal field theory on {${\bf P}^1$} and
  monodromy representations of the braid group.
\newblock {\em Lett. Math. Phys.}, 13(4):303--312, 1987.

\bibitem[Tsu]{ts}
Y. Tsuchimoto.
\newblock On the coordinate-free description of the conformal blocks.
\newblock {\em J. Math. Kyoto Univ.}, 33(1):29--49, 1993.

\bibitem[TUY]{tuy}
A. Tsuchiya, K. Ueno, and Y. Yamada.
\newblock Conformal field theory on universal family of stable curves with
  gauge symmetries.
\newblock In {\em Integrable systems in quantum field theory and statistical
  mechanics}, volume~19 of {\em Adv. Stud. Pure Math.}, pages 459--566.
  Academic Press, Boston, MA, 1989.

\bibitem[Uen]{UenoTheta}
K. Ueno.
\newblock On conformal field theory.
\newblock In {\em Vector bundles in algebraic geometry ({D}urham, 1993)},
  volume 208 of {\em London Math. Soc. Lecture Note Ser.}, pages 283--345.
  Cambridge Univ. Press, Cambridge, 1995.

\bibitem[Zhu]{zhu}
Y. Zhu.
\newblock Modular invariance of characters of vertex operator algebras.
\newblock {\em J. Amer. Math. Soc.}, 9(1):237--302, 1996.

\end{thebibliography}

\end{document}